\documentclass{amsart}

\usepackage{amsfonts,amsmath,amscd,amssymb,amsthm,graphicx,amsrefs}
\usepackage{enumerate,hyperref,perpage,url}

\usepackage[margin=2.2cm, a4paper]{geometry}

\theoremstyle{plain}
\newtheorem{thm}{\textbf{Theorem}}[section]

\newtheorem{lemma}[thm]{\textbf{Lemma}}
\newtheorem{prop}[thm]{\textbf{Proposition}}

\newtheorem{remark}{\textbf{Remark}}
\newtheorem{defn}{\textbf{Definition}}

\numberwithin{equation}{section}

\MakePerPage{footnote} \allowdisplaybreaks 

\newcommand{\norm}[1]{\left|\!\left|{#1}\right|\!\right|}

\newcommand\Id{\mathrm{Id}}

\newcommand\N{\mathbb{N}}

\newcommand\R{\mathbb{R}}

\newcommand\FT{\mathcal{F}_{h}}

\title[$L^{p}$ estimates for joint quasimodes, HOC in two dimensions]{ $L^{p}$ estimates for joint quasimodes of semiclassical pseudodifferential operators whose characteristic sets have $k$th order contact}

\author{Melissa Tacy}

\email{mtacy@maths.otago.ac.nz}

\address{Department of Mathematics and Statistics, University of Otago}

\begin{document}

  \begin{abstract}
In this paper we develop $L^{p}$ estimates for functions $u$ which are joint quasimodes of semiclassical pseudodifferential operators $p_{1}(x,hD)$ and $p_{2}(x,hD)$ whose characteristic sets meet with $k$th order contact, $k\geq{}1$. As part of the technical development we use  Fourier integral operators to adapt a flat wavelet analysis to the curved level sets of $p_{1}(x,\xi)$.
\end{abstract}
\maketitle

Let $(M,g)$ be a two dimensional compact, boundaryless Riemannian manifold and $p_{1}(x,hD)$, $p_{2}(x,hD)$ be two semiclassical pseudodifferential operators $L^{2}(M)\to L^{2}(M)$. In this paper we consider the question of the concentration properties of a function $u$ that approximately solves
$$p_{1}(x,hD)u=0\quad\text{and}\quad p_{2}(x,hD)u=0.$$
In particular we ask about the growth rate of $\norm{u}_{L^{p}}$ compared to $\norm{u}_{L^{2}}$. A key example to keep in mind is when one of the equations requires that $u$ be an approximate solution to the eigenfunction equation,
\begin{equation}-\Delta_{g} u=\lambda^{2}u.\label{eigeq}\end{equation}
Equation \eqref{eigeq} can be converted to a semiclassical equation by dividing through by $\lambda^{2}$ and setting $h=\lambda^{-1}$. Then, we require that $u$ satisfies the semiclassical equation
$$(-h^{2}\Delta_{g}-1)u=0.$$
In \cite{S88} Sogge shows that eigenfunctions (in fact more generally spectral clusters) of the Laplacian obey
\begin{equation}\norm{u}_{L^{p}}\lesssim \lambda^{\delta(p)}\norm{u}_{L^{2}}\label{sogge}\end{equation}
where
$$\delta(p)=\begin{cases}
\frac{1}{2}-\frac{2}{p}&6\leq{}p\leq\infty\\
\frac{1}{4}-\frac{1}{2p}&2\leq p\leq 6.\end{cases}$$
Koch, Tataru and Zworski \cite{koch} extend this result to approximate solutions of any semiclassical equation $p(x,hD)u=0$ where the symbol $p(x,\xi)$ displays sufficiently Laplace-like behaviour. These results are sharp in the sense that there exist examples that saturate the estimates given by \eqref{sogge}. 

In his letter to Morawetz \cite{S06} Sarnak poses the question of potential improvements when $u$ is (in addition to being a Laplacian eigenfunction) an eigenfunction of $r$ other differential operators. He obtains $L^{\infty}$ results under the assumption that $M$ is a rank $r$ symmetric space. Marshall then \cite{M16} extends Sarnak's result to a full set of $L^{p}$ estimates. 

In \cite{Tacy19} Tacy examines this problem from the perspective of contact between the characteristic hypersurfaces $\{\xi\mid p_{j}(x,\xi)=0\}$ and obtains sharp results under the condition that if $\nu_{j}(x,\xi)$ is the normal to $\{\xi\mid p_{j}(x,\xi)=0\}$ then $\nu_{1},\dots,\nu_{r}$ are linearly independent. In the two dimensional case this means that, for all $x$, the sets $\{\xi\mid p_{1}(x,\xi)=0\}$ and $\{\xi\mid p_{2}(x,\xi)=0\}$ meet with order $0$ contact. The results of \cite{Tacy19} give that in that case
$$\norm{u}_{L^{p}}\lesssim \norm{u}_{L^{2}}\quad\forall p\geq{}2.$$
In this paper we address the case where $\{\xi\mid p_{1}(x,\xi)=0\}$ and $\{\xi\mid p_{2}(x,\xi)=0\}$ meet with higher order contact. 

 Before we state our main theorem, let us consider first what estimates we might reasonably expect. We study the flat model example where $p_{1}(x,\xi)=|\xi|^{2}-1$ and $p_{2}(x,\xi)=p_{2}(\xi)$ is a smooth curve such that $\{p_{2}(\xi)=0\}$ intersects the circle at $(1,0)$ with order $k$ contact. We will assume that $u$ is an order $h$ joint quasimode, that is
 $$\norm{(-h^{2}\Delta_{\R^{n}}-1)u}_{L^{2}}\lesssim h\norm{u}_{L^{2}}\quad\text{and}\quad\norm{p_{2}(hD)u}_{L^{2}}\lesssim h\norm{u}_{L^{2}}.$$
 This level of error is a very natural choice (as discussed in Section \ref{sec:QM}). We will use the Fourier transform method to analyse such functions $u$. In particular we work with the semiclassical Fourier transform
 $$\FT[u](\xi)=\frac{1}{(2\pi h)^{n/2}}\int e^{-\frac{i}{h}\langle x,\xi\rangle}u(x)dx.$$
 With this normalisation $\FT$ is still an isometry on $L^{2}$ and has the property that
 $$\FT[hD_{x_{i}}u]=\xi_{i}\FT[u]\quad\text{and}\quad \FT[p(hD)u]=p(\xi)\FT[u].$$
 So from the first semiclassical equation we have that
 $$(|\xi|^{2}-1)\FT[u]=O_{L^{2}}(h).$$
 That is $\FT[u]$ should live, predominately, in the annulus of width $h$ around $|\xi|^{2}=1$.  The second equation gives us that
 $$p_{2}(\xi)\FT[u]=O_{L^{2}}(h).$$
\begin{figure}[h!]
  \begin{minipage}[l]{0.45\textwidth}
  \begin{center}
    \includegraphics[scale=0.4]{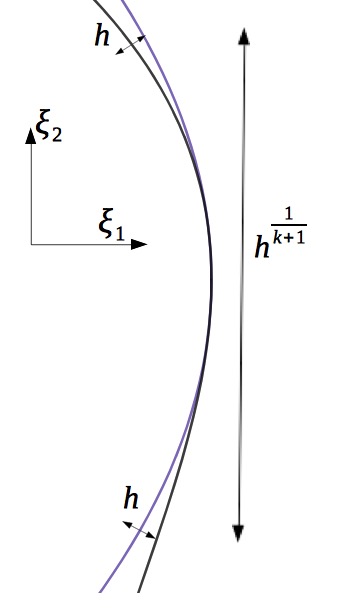}
    \end{center}
  \end{minipage}\hfill
  \begin{minipage}[l]{0.55\textwidth}
    \caption{
      The support of $\FT[u]$ must be within an order $h$ region of both curves
    } \label{fig:3ordercont}
  \end{minipage}
\end{figure}
So in addition $\FT[u]$ cannot be too large away from the set $p(\xi)=0$. We have assumed that $\{p_{1}(\xi)=0\}$ and $\{p_{2}(\xi)=0\}$ meet at a single point where they have $k$th order contact. The order of the contact controls how close the curves are to each other. If we write them locally as graphs around $\xi_{2}=0$,
$$\xi_{1}=a_{1}(\xi_{2})$$
and
$$\xi_{1}=a_{2}(\xi_{2})$$
then to have $k$th order contact the first $k$ derivatives must agree.  Once $|a_{1}(\xi_{2})-a_{2}(\xi_{2})|>ch$ it is impossible for the Fourier transform of $u$ to be simultaneously located in a order $h$ region of both curves. Therefore, since the contact is of order $k$, we would expect $\FT[u]$ to be small outside the region $|\xi_{2}|\leq h^{\frac{1}{k+1}}$, see Figure \ref{fig:3ordercont}.
 
 Now let's consider what kinds of quasimodes we can construct within such restrictions. We will use the family of examples, $T^{h}_{\alpha}(x)$ from \cite{Tacy18} given by
 $$T^{h}_{\alpha}(x)=\frac{h^{-\frac{3+\alpha}{2}}e^{\frac{i}{h}x_{1}}}{2\pi}\int_{\R^{2}}e^{\frac{i}{h}\left(x_{1}(\xi_{1}-1)+x_{2},\xi_{2}\right)}\chi_{\alpha}(\xi)d\xi$$
 where  $\omega_{0}$ corresponds to the unit vector in the $\xi_{1}$ direction and
 $$\chi_{\alpha}(\xi)=\chi_{\alpha}(r,\omega)=\begin{cases}1& |r-1|<h\text{ and }|\omega-\omega_{0}|<h^{\alpha}\\
 0 &\text{otherwise.}\end{cases}$$

\begin{figure}[h!]
    \centering
    \begin{minipage}{0.45\textwidth}
        \centering
        \includegraphics[scale=0.4]{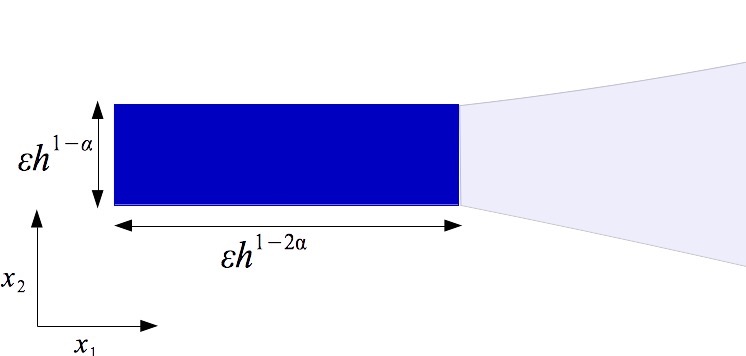}\label{fig:talpha}
         \caption{$T^{h}_{\alpha}$ saturates its $L^{\infty}$ estimates in an $\epsilon h^{1-2\alpha}\times \epsilon h^{1-\alpha}$ region.}
    \end{minipage}\hfill
    \begin{minipage}{0.45\textwidth}
        \centering
        \includegraphics[scale=0.4]{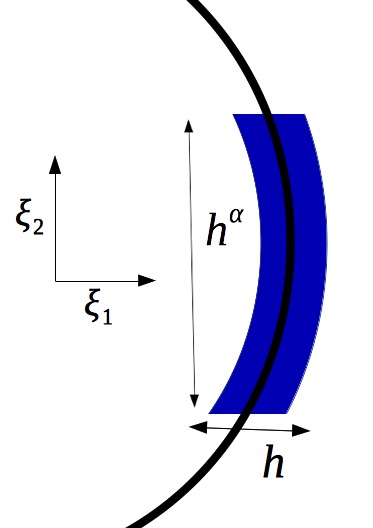}\label{fig:falpha}
     \caption{The support of $\FT\left[T^{h}_{\alpha}]\right]$ has an $h^{\alpha}$ spread.}
    \end{minipage}
\end{figure}

 Clearly $T^{h}_{\alpha}(x)$ is an order $h$ quasimode of $(-h^{2}\Delta_{\R^{2}}-1)$.  In \cite{Tacy18} it is demonstrated that $\norm{T^{h}_{\alpha}}_{L^{2}}=1$ and for some $\epsilon>0$ sufficiently small,
 $$|T^{h}_{\alpha}(x)|>ch^{-\frac{1}{2}+\frac{\alpha}{2}}$$
 on a $\epsilon h^{1-2\alpha}\times\epsilon h^{1-\alpha}$ tube about $(0,0)$ with long direction $e_{1}$. The contact condition means that we should only consider the behaviour of the examples where $h^{\alpha}\leq h^{\frac{1}{k+1}}$. Since the case where $k=0$ is treated in \cite{Tacy19} we focus on $k\geq{}1$. Notice that $h^{\frac{1}{2}}\leq h^{\frac{1}{k+1}}$ for  $k\geq 1$. Therefore the $T^{h}_{1/2}(x)$ are always joint quasimodes. These examples have the property that
 $$\norm{T_{1/2}^{h}}_{L^p}>c h^{-\frac{1}{4}+\frac{1}{2p}}.$$
 So for $2\leq p\leq 6$ we can expect no improvement over Sogge's estimate. To saturate the high $p$ estimates one usually picks $\alpha$ so that the support of $\FT[u]$ is spread through the largest possible region. In this case that is $\alpha=\alpha_{k}$ so that $h^{\alpha_{k}}=h^{\frac{1}{k+1}}$. That is $\alpha_{k}=\frac{1}{k+1}$ (if $\alpha$ is any smaller than $\alpha_{k}$, $T_{\alpha}^{h}(x)$ will not be a good quasimode of $p_{2}(hD)$). Now
\begin{align}
\norm{T_{\alpha_{k}}^{h}}_{L^{p}}&>ch^{-\frac{1}{2}+\frac{1}{2(k+1)}}\left(h^{1-\frac{2}{k+1}}h^{1-\frac{1}{k+1}}\right)^{\frac{1}{p}}\nonumber\\
&=ch^{-\frac{1}{2}+\frac{2}{p}+\frac{1}{k+1}\left(\frac{1}{2}-\frac{3}{p}\right)}.\label{highp}\end{align}
This tells us that we could never expect a better upper bound than \eqref{highp}. In Theorem \ref{thm:main} we see indeed that this is the highest growth rate possible. 

\begin{thm}\label{thm:main}
Suppose $u$ is a semiclassically localised,  strong joint  quasimode of order $h$ (see Definitions \ref{defn:localised} and \ref{defn:strongqm} for the definitions of semiclassically localised and strong joint quasimode) for a pair of semiclasscial pseudodifferential operators $p_{1}(x,hD), p_{2}(x,hD)$ where the symbols $p_{j}(x,\xi)$ obey the following admissibility conditions.
\begin{itemize}
\item For each $x_{0}$ and $j=1,2$ the set $\{\xi\mid p_{j}(x_{0},\xi)\}$ is a smooth hypersurface.
\item For each $x_{0}$ the sets $\{\xi\mid p_{j}(x_{0},\xi)\}$ meet at a single point $\xi_{0}$ and at that point have $k$th order contact.
\item There is some $j$ such that for all $x_{0}$, the sets $\{\xi\mid p_{j}(x_{0},\xi)=0\}$ have non-degenerate second fundamental form.

\end{itemize}
Then
$$\norm{u}_{L^{p}}\lesssim h^{-\delta(p,k)}\norm{u}_{L^{2}},$$
where
$$\delta(p,k)=\begin{cases}
-\left(\frac{1}{2}-\frac{2}{p}\right)+\frac{1}{k+1}\left(\frac{1}{2}-\frac{3}{p}\right)&6\leq p\leq\infty\\
-\frac{1}{4}+\frac{1}{2p}&2\leq p\leq 6.\end{cases}$$

\end{thm}

\begin{remark}
It is necessary only to consider the case where $u$ is a strong joint quasimode of order $h$ of two operators. Suppose $u$ is a strong joint quasimode of order $h$  of $r$ operators $p_{i}(x,hD)$ with orders of contact $k_{ij}$. Then the best estimate for the $L^{p}$ norms of $u$ will come from applying Theorem \ref{thm:main} to the pair $p_{i}(x,hD),p_{j}(x,hD)$ with smallest contact order $k_{ij}$. 
\end{remark}

In this paper we will many times rely on a number of standard results from semiclassical analysis. In particular those regarding the composition and invertibility properties of semiclassical pseudodifferential operators and the development of parametrix constructions for propagators. The relevant results are listed in Appendix A for the convenience of the reader. 

\section{Combining wavelets with Fourier integral operators}

To obtain the results of Theorem \ref{thm:main} we develop a way to combine the theory of Fourier integral operators with that of wavelet analysis. The basic idea dates back to some of the earliest successes in microlocal analysis. In \cite{Feff83} Fefferman describes ``the algorithm of the 70s'' for understanding the $L^{2}$ theory of PDEs with variable coefficients. Consider a partial differential equation
$$Pu=0\quad\text{where}\quad P=\sum_{|\gamma|\leq{}N}c_{\gamma}(x)D^{\gamma}.$$
This can be expressed as a pseudodifferential equation $p(x,D)u=0$ where
$$p(x,D)u=\frac{1}{(2\pi)^{n}}\int e^{i\langle x-y,\xi\rangle}p(x,\xi)u(y)d\xi dy$$
and
$$p(x,\xi)=\sum_{|\gamma|\leq{}N}c_{\gamma}(x)\xi^{\gamma}.$$
Standard considerations about the invertibility of pseudodifferential operators ensure that if $Pu=0$, the microsupport of $u$ must lie inside the set $\{p(x,\xi)=0\}$. The simplest manifestation of ``the algorithm of the 70s'' is the case where $\{p(x,\xi)=0\}$ is a smooth hypersurface. The idea is to (after localisation) straighten $\{p(x,\xi)=0\}$ out to become the hypersurface $\{\xi_{1}=0\}$. The solutions to $D_{x_{1}}u=0$ are simple to understand, this algorithm allows  information to be carried back to the more complicated $Pu=0$. In particular one develops a unitary Fourier integral operator $W$ that has the property that
$$D_{x_{1}}W=Wp(x,D).$$
Then if $v=Wu$ and $Pu=0$  we have that $v$ is a solution to $D_{x_{1}}v=0$ (or an approximate solution if $u$ is only an approximate solution to $Pu=0$). Exactly the same technique can be applied to semiclassical pseudodifferential operators, except in this case we have
$$hD_{x_{1}}W=Wp(x,hD).$$
Since $W$ is unitary the $L^{2}$ theory for $u$ follows from the $L^{2}$ theory for $v$. However it is immediate that the $L^{p}$ theory cannot so directly follow. Take for example $p(x,\xi)=\xi_{1}-\xi_{2}^{2}$ localised near $(0,0)$. The hypersurface $\{\xi_{1}-\xi_{2}^{2}=0\}$ can be flattened out to $\{\xi_{1}=0\}$ by a suitable Fourier integral operator. Note that $\{\xi_{1}=0\}$ is flat and $\{\xi_{1}-\xi_{2}^{2}=0\}$ is curved. Classical theory on the Fourier restriction/extension problem (see for example Stein \cite{S} Chapter 9 for an authoritative source) tell us that the $L^{p}$ theory of solutions to $Pu=0$ depend crucially on the curvature of the characteristic set. This information about the curvature is encoded in the $L^{2}\to L^{p}$ mapping properties of $W^{-1}$. 

In this paper we will use ``the algorithm of the 70s'' but incorporate some wavelet theory. Suppose we have an operator $W$ so that
$$hD_{x_{1}}W=Wp_{1}(x,hD)$$
and we assume that $p_{1}(x,\xi)$ is the symbol that satisfies the curvature condition of Theorem \ref{thm:main}. Then if $v=Wu$, we have that $v$ is an order $h$ quasimode of $hD_{x_{1}}u=0$. That is 
$$\norm{hD_{x_{1}}v}_{L^{2}}\lesssim h\norm{u}_{L^{2}}.$$
Therefore the semiclassical Fourier transform of $v$, $\FT[v]$ must be supported near $\{\xi_{1}=0\}$. This property makes $v$ very suitable for a wavelet decomposition in the $x_{1}$ variable. To that end let $f$ be smooth compactly supported function with
$$\int f(\tau)d\tau=0\quad \text{and}\quad C_{f}=\int\frac{|\hat{f}(\xi)|^{2}}{|\xi|}d\xi<\infty.$$
Then using the continuous wavelet transform we can write
$$v=\frac{1}{C_{f}}\int \frac{1}{|a|^{5/2}}X_{v}(a,b,x_{2})f(a^{-1}(x_{1}-b))dadb$$
where
$$X_{v}(a,b,x_{2})=\frac{1}{|a|^{1/2}}\int f(a^{-1}(y_{1}-b))v(y_{1},x_{2})dy_{1}.$$
Since $v$ is a good approximate solution to $hD_{x_{1}}v=0$ we expect that the main contribution to the integral comes from where $|a|\sim 1$. We write
$$u=W^{-1}v=\frac{1}{C_{f}}\int \frac{1}{|a|^{5/2}}W^{-1}\left[X_{v}(a,b,x_{2})f(a^{-1}(x_{1}-b))\right]dadb.$$
We also know that $u$ is a quasimode of $p_{2}(x,hD)$. Heuristically $W$ straightens out $\{p_{1}(x,\xi)=0\}$ to become $\{\xi_{1}=0\}$ so we would expect that $\{p_{2}(x,\xi)=0\}$ transforms to a curve that meets $\{\xi_{1}=0\}$ with order $k$ contact (as depicted in Figure \ref{fig:wacts}). We may as well assume that this intersection point is $(0,0)$. Then near $(0,0)$, $p_{2}(x,\xi)$ should have local form $p_{2}(x,\xi)=\xi^{k+1}g(x,\xi)$ where $g(0,0)\neq{}0$. So we would expect that $v$ is an order $h$ quasimode of $h^{k+1}D^{k+1}_{x_{2}}$ in addition to $hD_{x_{1}}=0$. 
\begin{figure}[h!]
  \begin{minipage}[c]{0.45\textwidth}
  \begin{center}
    \includegraphics[scale=0.4]{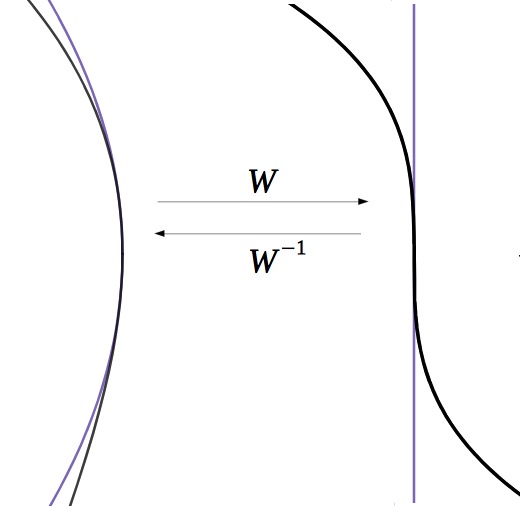}
    \end{center}
  \end{minipage}\hfill
  \begin{minipage}[l]{0.55\textwidth}
    \caption{The set $\{p_{1}(x,\xi)=0\}$ is flattened out to become $\{\xi_{1}=0\}$. The order of contact between the curves is preserved} \label{fig:wacts}
  \end{minipage}
\end{figure}
This is in fact what happens (see Proposition \ref{prop:newops}). Therefore to incorporate the fact that $h^{k+1}D^{k+1}_{x_{2}}v=O_{L^{2}}(h)$ we take a Fourier transform of $v$ in $x_{2}$ and dyadically decompose into regions $|\xi_{2}|\sim 2^{j}h^{\frac{1}{k+1}}$. 

Another way to think of this is to define
\begin{align*}
\FT[X](a,b,\xi_{2})&=\FT[X(a,b,\cdot)](\xi_{2})\\
&=\frac{1}{(2\pi h)^{\frac{1}{2}}}\int e^{-\frac{i}{h}\langle x_{2},\xi_{2}\rangle}\frac{1}{|a|^{1/2}}f(a^{-1}(x_{1}-b))v(x_{1},x_{2})dx_{2}dx_{2}.\end{align*}
Then
\begin{align}
u&=W^{-1}v\nonumber\\
&=W^{-1}\left[\frac{1}{(2\pi h)^{\frac{1}{2}}}\int e^{\frac{i}{h}\langle x_{2},\xi_{2}\rangle}\frac{1}{|a|^{5/2}}\FT[X](a,b,\xi_{2})f(a^{-1}(x_{1}-b))dadbd\xi_{2}\right]\nonumber\\
&=\frac{1}{(2\pi h)^{\frac{1}{2}}}\int \frac{1}{|a|^{5/2}}\FT[X](a,b,\xi_{2})W^{-1}\left[e^{\frac{i}{h}\langle x_{2},\xi_{2}}f(a^{-1}(x_{1}-b))\right]dadbd\xi_{2}\nonumber\\
&=\frac{1}{(2\pi h)^{\frac{1}{2}}}\int \frac{1}{|a|^{5/2}}\FT[X](a,b,\xi_{2})\psi_{a,b,\xi_{2}}(x_{1},x_{2})dadbd\xi_{2}\label{decomp}\end{align}
where
$$\psi_{a,b,\xi_{2}}=W^{-1}\left[e^{\frac{i}{h}\langle \cdot ,\xi_{2}\rangle}f(a^{-1}(\cdot -b))\right].$$
Here we are seeing \eqref{decomp} as synthesis for $u$ in terms of $\psi_{a,b,\xi_{2}}$. This can be understood as a specific example of more general phenomena. Suppose $\Lambda$ is a parameter space and $T_{v}(\lambda)$ is an analysis operator
$$T_{v}(\lambda)=\langle v,\phi_{\lambda}\rangle$$
with synthesis
$$v=\int T_{v}(\lambda)\phi_{\lambda}(x)d\mu(\lambda)$$
for some functions $\phi_{\lambda}$ and measure $\mu$ on $\Lambda$. Then if
$$(T\circ W)_{u}(\lambda)=T_{Wu}(\lambda)=\langle Wu,\phi_{\lambda}\rangle=\langle u,W^{\star}\phi_{\lambda}\rangle=\langle u,W^{-1}\phi_{\lambda}\rangle.$$ 
We can see $(T\circ W)_{u}(\lambda)$ as an analysis operator for $u$ with synthesis
$$u=\int (T\circ W)_{u}(\lambda)\psi_{\lambda}(x)d\mu(\lambda)$$
where $\psi_{\lambda}=W^{-1}(\phi_{\lambda})$. Therefore we may take  an appropriate analysis/synthesis of $L^{2}(\R^{n})$ and use $W^{-1}$ to produce a new analysis/synthesis that is adapted to the operator $p(x,hD)$.

\section{Quasimodes and joint quasimodes}\label{sec:QM}
The focus of this paper is on $u$ that satisfy
$$(-h^{2}\Delta_{g}-1)u=0$$
or some similar semiclassical equation. By working in coordinate charts and associating each patch with a patch on $\R^{n}$, we can write the operator $-h^{2}\Delta_{g}-1$ as a semiclassical quantisation of a symbol $p(x,\xi)$ which has principal symbol $|\xi|_{g}^{2}-1$. Here we use the left quantisation
\begin{equation}(-h^{2}\Delta_{g}-1)u=p(x,hD)u=\frac{1}{(2\pi h)^{n}}\int e^{\frac{i}{h}\langle x-y,\xi\rangle}p(x,\xi)u(y)dyd\xi.\label{pdef}\end{equation}
Since we must localise to make sense of \eqref{pdef} it is reasonable to only consider those $u$ which are semiclassically localised in phase space.

\begin{defn}\label{defn:localised}
We say that $u$ is semiclassically localised if there is a smooth, compactly supported function $\chi:T^{\star}M\to\R$ so that
$$u=\chi(x,hD)u+O(h^{\infty}).$$
\end{defn}

Localisation means that we will need to work with approximate solutions rather than exact ones. The commutation property for semiclassical pseudodifferential operators, Proposition \ref{prop:commute}, tells us that even if $p(x,hD)u=0$,
$$p(x,hD)\chi(x,hD)u=\chi(x,hD)p(x,hD)u+hr(x,hD)u=hr(x,hD)u$$
that is $\chi(x,hD)u$ is only an approximate solution with error $O_{L^{2}}(h\norm{u}_{L^{2}})$. Therefore it make sense to work with approximate solutions (quasimodes) with error $O_{L^{2}}(h\norm{u}_{L^2})$ from the start. 

\begin{defn}\label{defn:qm}
We say that $u$ is a quasimode of  order $h^{\beta}$ (sometimes written as $O_{L^{2}}(h^{\beta})$ or $O(h^{\beta})$)  of $p(x,hD)$ if
$$\norm{p(x,hD)u}_{L^{2}}\lesssim{}h^{\beta}\norm{u}_{L^{2}}.$$
If $u$ is a joint quasimode of order $h^{\beta}$  of $p_{1}(x,hD)$ and $p_{2}(x,hD)$ then
$$\norm{p_{i}(x,hD)u}_{L^{2}}\lesssim{}h^{\beta}\norm{u}_{L^{2}}\quad i=1,2.$$
\end{defn}

Definition \ref{defn:qm} is enough to produce the $L^{p}$ estimates for quasimodes considered in \cite{koch}, \cite{tacy09} and \cite{HTacy}. However for this work as in \cite{Tacy19} we will need a slightly stronger kind of quasimode. As discussed in \cite{Tacy19} if we start with a exact solution $u$ we could produce a quasimode $v$ by taking
$$v=u+hf$$
for some function $\norm{f}_{L^{2}}=1$. However such  examples are rather artificial. To avoid this, as in \cite{Tacy19}, we will work with strong quasimode. Strong quasimodes have the property  that repeated application of $p(x,hD)$ continues to improve the quasimode error. 

\begin{defn}\label{defn:strongqm}
We say that $u$ is a strong quasimode of order $h^{\beta}$ ($O^{str}_{L^{2}}(h^{\beta})$ or $O^{str}(h^{\beta})$)  of $p(x,hD)$ if
$$\norm{p^{M}(x,hD)u}_{L^{2}}\lesssim{}h^{\beta M}\norm{u}_{L^{2}},\quad M=1,2,\dots$$
If $u$ is a  strong joint quasimode of order $h^{\beta}$ of $p_{1}(x,hD)$ and $p_{2}(x,hD)$ then
\begin{equation}\norm{p^{M_{1}}_{1}(x,hD)\circ p^{M_{2}}_{2}(x,hD)u}_{L^{2}}\lesssim{}h^{\beta (M_{1}+M_{2})}\norm{u}_{L^{2}},\quad i=1,2\; M_{i}=1,2,\dots\label{jqmdef}\end{equation}
In some cases (where the comparison to $\norm{u}_{L^{2}}$ is important) we will write $O^{str}_{L^{2}}(h^{\beta}\norm{u}_{L^{2}})$ to indicate \eqref{jqmdef} holds. 
\end{defn}

Clearly an exact solution $u$,
$$p(x,hD)u=0$$
is a strong quasimode. As discussed in \cite{Tacy19} spectral clusters (a major example of quasimodes) are also strong quasimodes. A spectral cluster of window width $W=W(\lambda)$ is given by,
$$u=\sum_{\lambda_{j}\in[\lambda,\lambda-W(\lambda)]}c_{j}\phi_{j}\quad \text{with}\quad-\Delta \phi_{j}=\lambda_{j}^{2}\phi_{j}.$$
Such functions are strong order $W(h^{-1})h$ quasimodes of $(-h^{2}\Delta_{g}-1)$ when $h=\lambda^{-1}$.

We have seen that the commutation identity implies that the property of being an order $h$ quasimode is preserved under localisation.  That is, if $u$ is an order $h$ quasimode of $p(x,hD)$, $\chi(x,hD)u$ is also an $O_{L^{2}}(h)$ quasimode of $p(x,hD)$. This property also holds for strong joint quasimodes (see Proposition  1.4 in \cite{Tacy19}). If $u$ is a  strong joint quasimode of order $h$ of $p_{1}(x,hD)$ and $p_{2}(x,hD)$ then $\chi(x,hD)u$ is also a  strong joint quasimode of order $h$ of $p_{1}(x,hD)$ and $p_{2}(x,hD)$. Similarly, in Lemma \ref{lem:jqm}, we will see that if $u$ is a  strong joint quasimode of order $h$ of $p_{1}(x,hD)$ and $p_{2}(x,hD)$ it is also a  strong joint quasimode of order $h$ of $p_{1}(x,hD)$ with linear combinations of $p_{2}(x,hD)$ and $p_{1}(x,hD)$.

\begin{lemma}\label{lem:jqm}
Suppose $u$ is a strong joint  quasimode of order $h$ for $p_{1}(x,hD)$ and $p_{2}(x,hD)$ and $e_{1}(x,hD)$, $e_{2}(x,hD)$ are any semiclassical pseudodifferential operator with smooth symbols. Then $u$ is also a  strong joint quasimode  of order $h$ for any pairs consisting of elements of $$\{p_{1}(x,hD), p_{2}(x,hD), \left(p_{1}(x,hD)e_{1}(x,hD) + p_{2}(x,hD)e_{2}(x,hD)\right)\}.$$
\end{lemma}

\begin{proof}
Given any composition $q(x,hD)p(x,hD)$ the semiclassical calculus allows us to commute with loss of an $O_{L^{2}}(h)$ term (Proposition \ref{prop:com}) so
$$q(x,hD)p(x,hD)=p(x,hD)q(x,hD)+hr(x,hD)$$
where $r(x,hD)$ maps $L^{2}\to L^{2}$ with norm bounded independent of $h$.  Expanding 
$$(p_{1}(x,hD)e_{1}(x,hD) + p_{2}(x,hD)e_{2}(x,hD))^{M}$$
 and commuting terms as necessary we have that
$$(p_{1}(x,hD)e_{1}(x,hD)+p_{2}(x,hD)e_{2}(x,hD))^{M}=\sum_{m=0}^{M}\sum_{j=0}^{M-m}h^{m}r_{j,m}(x,hD)p_{1}(x,hD)^{j}p_{2}(x,hD)^{M-m-j}$$
where all the $r_{j,m}(x,hD)$ map $L^{2}\to L^{2}$ with norm bounded independent of $h$. So $u$ is indeed a  strong joint  quasimode of order $h$ of $ p_{1}(x,hD)e_{1}(x,hD)+p_{2}(x,hD)e_{2}(x,hD) $ and either $p_{1}(x,hD)$ or $p_{2}(x,hD)$. 
\end{proof}

Now suppose we have the hypotheses of Theorem \ref{thm:main}. That is, $u$ is a  strong joint quasimode of order $h$ of $p_{1}(x,hD)$ and $p_{2}(x,hD)$. We will see that, as in \cite{Tacy19}, it is enough to estimate $\norm{\chi(x,hD)u}_{L^{p}}$ where $\chi(x,\xi)$ is a smooth function supported in a small (but $h$ independent) region about a point $(x_{0},\xi_{0})$,
$$(x_{0},\xi_{0})\in\bigcap_{i=1}^{2}\left\{(x,\xi)\mid p_{i}(x,\xi)=0\right\}.$$
To see that such a localisation is valid  consider $\chi(x,hD)u$ where $\chi(x,\xi)$ is supported near a point $(x_{0},\xi_{0})$ such that $p_{i}(x_{0},\xi_{0})\neq{}0$.   From Proposition \ref{prop:invert} we know that if $|p_{i}(x,\xi)|>c>0$, the operator $p_{i}(x,hD)$ is invertible and its inverse $(p_{i}(x,hD))^{-1}$ has bounded mapping norm $L^{2}\to{}L^{2}$. By choosing the support of $\chi$ small enough we may assume that $p_{i}(x,\xi)$ is bounded away from zero on the support of $\chi$ and therefore so is $p^{M}_{i}(x,\xi)$. Proposition \ref{prop:com} tells us that $p^{M}_{i}(x,\xi)$ is the principal symbol of $p^{M}_{i}(x,hD)$ so by Proposition \ref{prop:invert} we can produce an inverse $(p^{M}_{i}(x,hD))^{-1}$. Therefore if 
$$p^{M}_{i}(x,hD)\chi(x,hD)u=h^{M}f\quad\norm{f}_{L^{2}}\lesssim \norm{u}_{L^{2}},$$
we can invert $p_{i}(x,hD)$ to obtain
$$\chi(x,hD)u=h^{M}(p^{M}(x,hD))^{-1}f$$
and
 $$\norm{\chi(x,hD)u}_{L^{2}}\lesssim h^{M}\norm{u}_{L^{2}}.$$ 
 By applying semiclassical Sobolev estimates \cite{Zworski12} we obtain
$$\norm{\chi(x,hD)u}_{L^{p}}\lesssim{}h^{-\frac{n}{2}+\frac{n}{p}+M}\norm{u}_{L^{2}}.$$
Choosing $M$ large enough we obtain better estimates than those of Theorem \ref{thm:main}. So we need only consider $\chi(x,hD)u$ where $\chi(x,\xi)$ is supported in a neighbourhood of some point $(x_{0},\xi_{0})$ where both of  the $p_{i}(x_{0},\xi_{0})=0$. We may as well assume this point is $(0,0)$.

Assume that $p_{1}(x,\xi)$ satisfies the curvature condition.   Since $\{\xi\mid p_{1}(x_{0},\xi)\}$ is a smooth hypersurface in $\R^{2}$ we may assume that (after a suitable change of coordinate system) that $\partial_{\xi_{1}}p_{1}(0,0)\neq 0$. By taking a suitable localisation we can extend this to the property that $\partial_{\xi_{1}}p_{1}(x,\xi)$ is bounded away from zero on the support of $\chi(x,\xi)$. Therefore we may write
$$p_{1}(x,\xi)=e_{1}(x,\xi)(\xi_{1}-a(x,\xi_{2}))$$
where $|e_{1}(x,\xi)|>c>0$. So the invertibility of $e_{1}(x,hD)$ ensures that $u$ is also a quasimode of $(hD_{x_{1}}-a(x,hD_{x_{2}}))$. Since the characteristic sets of meet with contact of at least order one at $(0,0)$ we can also factorise $p_{2}(x,\xi)$ as
$$p_{2}(x,\xi)=e_{2}(x,\xi)(\xi_{1}-q(x,\xi_{2}))$$
where $|e_{2}(x,\xi)|>c>0$. Therefore (since $e_{2}(x,hD)$ is invertible) $u$ is a quasimode of $(hD_{x_{1}}-q(x,hD_{x_{2}}))$.
Now let $W(x_{1}):L^{2}(\R) \to L^{2}(\R)$ be the operator so that
\begin{equation}hD_{x_{1}}W(x_{1})=-W(x_{1})a(x,hD_{x_{2}})\label{Wdef}\end{equation}
or equivalently 
$$(hD_{x_{1}}-a(x,hD_{x_{2}}))W^{\star}=0.$$
We can write (see Proposition \ref{prop:parametrix})
$$Wg=\frac{1}{2\pi h}\int e^{\frac{i}{h}\left( x_{2}\xi_{2}-\phi(x_{1},y_{2},\xi_{2})\right)}b(x_{1},y_{2},\xi_{2})g(y_{2})d\xi_{2}dy_{2}$$
where 
$$\partial_{x_{1}}\phi+a(x_{1},y_{2},\partial_{y_{2}}\phi)=0\quad\phi(0,y_{2},\xi_{2})=\langle y_{2},\xi_{2}\rangle$$
and
$$b(0,y_{2},\xi_0)=1.$$
It is a classical result (see Zworski \cite{Zworski12} Part 3, particularly Chapters 10-11) that
$$W^{\star}(x_{1})W(x_{1})=\Id\quad\text{and}\quad W(x_{1})W^{\star}(x_{1})=\Id.$$
Let 
\begin{equation}v(x_{1},x_{2})=W(x_{1})u(x_{1},\cdot)\label{vdef}\end{equation}
 then
$$u=W^{\star}v.$$
So to obtain Theorem \ref{thm:main} we need to to obtain a bound
$$\norm{W^{\star}v}_{L^{p}}\lesssim h^{-\delta(n,p,k)}\norm{u}_{L^{2}}.$$
First we see that $v$ is a strong joint quasimode of $hD_{x_{1}}$ and $h^{k+1}D^{k+1}_{x_{2}}$. 

\begin{prop}\label{prop:newops}
Suppose $W(x_{1})$ and $v(x_{1},x_{2})$ are given by \eqref{Wdef} and  \eqref{vdef} respectively. Then $v$ is a  strong joint  quasimode of order $h$ of the operators $hD_{x_{1}}$ and $h^{k+1}D^{k+1}_{x_{2}}$. 
\end{prop}

\begin{proof}

Note that
\begin{align*}
(hD_{x_{1}})^{M}v&=hD_{x_{1}}\left(W(x_{1})u(x_{1},\cdot)\right)\\
&=W(x_{1})\left(hD_{x_{1}}-a(x,hD_{x_{2}}\right)^{M}u.\end{align*}
Therefore since $W(x_{1})$ is unitary
$$\norm{(hD_{x_{1}})^{M}v}_{L^{2}_{x_{2}}}\lesssim \norm{(hD_{x_{1}}-a(x,hD_{x_{2}}))^{M}u}_{L^{2}_{x_{2}}}.$$
Using the fact that $v$ is localised to an $O(1)$ region gives us
$$\norm{(hD_{x_{2}})^{M}v}_{L^{2}}\lesssim h^{M}\norm{u}_{L^{2}}.$$
So $v$ is certainly a strong quasimode of order $h$ of $hD_{x_{1}}$. 

Recall that the factorisation of $p_{2}(x,\xi)$ ensures that $u$ is a quasimode of 
$$(hD_{x_{1}}-q(x,hD_{x_{2}}))u=0.$$
We have assumed that we are localised about the point $(0,0)$ where both $p_{1}(0,0)=0=p_{2}(0,0)$, which implies $a(0,0)=0=q(0,0)$. Recall  that $W^{\star}(x_{1})$ is the propagator for the time evolution equation
$$(hD_{x_{1}}-a(x,hD_{x_{2}}))W^{\star}(x_{1})=0$$
where $x_{1}$ acts as the time variable.  Define the classical system
$$\begin{cases}
\dot{x}_{2}(x_{1})=\partial_{\xi_{2}} a(x_{1},x_{2},\xi_{2})\\
x_{2}(0)=x_{2}\\
\dot{\xi}_{2}(x_{1})=-\partial_{x_{2}}a(x_{1},x_{2},\xi_{2})\\
\xi_{2}(0)=\xi_{2}\end{cases}$$
and let
\begin{align*}
\tilde{q}(x,\xi_{2})&=q(x_{1},x_{2}(x_{1}),\xi_{2}(x_{1}))\\
\tilde{a}(x,\xi_{2})&=a(x,x_{2}(x_{1}),\xi_{2}(x_{1})).\end{align*}
 Egorov's theorem tells us that
\begin{align*}
\tilde{q}(x,hD_{x_{2}})&=W(x_{1})q(x,hD_{x_{2}})W^{\star}(x_{1})+hr_{1}(x,hD)\\
&\text{and}\\
\tilde{a}(x,hD_{x_{2}})&=W(x_{1})a(x,hD_{x_{2}})
)W^{\star}(x_{1})+hr_{2}(x,hD),\end{align*}
where both $r_{i}(x,hD):L^2\to L^{2}$ are bounded independent of $h$. So if we let
$$\tilde{p}_{2}(x,\xi)=\xi_{1}+\tilde{a}(x,\xi_{2})-\tilde{q}(x,\xi_{2})$$
then
\begin{align*}(hD_{x_{1}})^{M_{1}}(\tilde{p}_{2}(x,hD))^{M_{2}}v&=(hD_{x_{1}})^{M_{1}}(\tilde{p}_{2}(x,hD))^{M_{2}-1}W(hD_{x_{1}}\!\!-\!q(x,hD_{x_{2}})u\\
&\qquad+h(hD_{x_{1}})^{M_{1}}p_{2}(x,hD)^{M_{2}-1}r_{1}(x,hD)W(x_{1})u\\
&\quad\vdots\\
&=\sum_{j=0}^{M_{2}}h^{j}(hD_{x_{1}})^{M_{1}}R_{j}(x,hD)W(hD_{x_{1}}-q(x,hD_{x_{2}}))^{M_{2}-j}u\\
&=\sum_{j=0}^{M_{2}}h^{j}(hD_{x_{1}})^{M_{1}-1}R_{j}(x,hD)W(hD_{x_{1}}\!\!-\!a(x,hD_{x_{2}}))(hD_{x_{1}}\!\!-\!q(x,hD_{x_{2}}))^{M_{2}-j}u\\
&\qquad+\sum_{j=0}^{M_{2}}h^{j+1}(hD_{x_{1}})^{M_{1}-1}R_{j,1}(x,hD)W(hD_{x_{1}}\!\!-\!q(x,hD_{x_{2}}))^{M_{2}-j}u\\
&\quad\vdots\\
&=\sum_{m=0}^{M_{1}}\sum_{j=0}^{M_{2}}h^{j+m}R_{j,m}(x,hD)(hD_{x_{1}}\!\!-\!a(x,hD_{x_{2}}))^{M_{1}-m}(hD_{x_{1}}\!\!-\!q(x,hD_{x_{2}})^{M_{2}-j}u\end{align*}
where all the pseudodifferential operators $R_{j,m}(x,hD)$ are bounded $L^{2}\to L^{2}$ independent of $h$. Therefore
$$\norm{(hD_{x_{1}})^{M_{1}}(\tilde{p}_{2}(x,hD))^{M_{2}}v}_{L^{2}}\lesssim h^{M_{1}+M_{2}}\norm{u}_{L^{2}}$$
and $v$ is a strong joint quasimode of order $h$  of $\tilde{p}_{2}(x,\xi)$ and $hD_{x_{1}}$. Since $a(0,0)=0=q(0,0)$ we have that $\tilde{p}_{2}(0,0)=0$.  Consider 
$$\frac{\partial^{r}}{\partial\xi_{2}^{r}}\tilde{p}_{2}(x,\xi_{2})=\frac{\partial^{r}(a-q)}{\partial\xi_{2}}(x_{1},x_{2}(t),\xi_{2}(t))\left(\frac{\partial\xi_{2}(x_{1})}{\partial\xi_{2}}\right)^{r}+R(x_{1},x_{2},\xi_{2})$$
where $R(x_{1},x_{2},\xi_{2})$ is a sum of terms all of which have a factor of
\begin{equation}\frac{\partial^{\gamma}x_{2}(x_{1})}{\partial\xi^{\gamma}_{2}}\Big|_{(x,\xi)=(0,0)}\label{0deriv1}\end{equation}
with $\gamma\leq{}r$ or a factor of
\begin{equation}\frac{\partial^{\beta}\xi_{2}(x_{1})}{\partial\xi_{2}^{\beta}}\Big|_{(x,\xi)=(0,0)}\label{0deriv2}\end{equation}
for $2\leq{}\beta\leq r$. Now
$$x_{2}(x_{1})=x_{2}+O(|x_{1}|)\quad\text{and}\quad \xi_{2}=\xi_{2}+O(|x_{1}|).$$ 
So at $x_{1}=0$ all of the factors of the form \eqref{0deriv1} and \eqref{0deriv2} are zero and
$$\frac{\partial\xi_{2}(x_{1})}{\partial\xi_{2}}\Big|_{x_{1}=0}=1.$$
Since the derivatives of $a$ and $q$ agree up to the $k$th derivative (but not at the $k+1$st derivative) we have that
$$\frac{\partial^{r}}{\partial\xi_{2}^{r}}\tilde{p}_{2}(x,\xi_{2})\Big|_{(0,0)}=\begin{cases}
0&1\leq r\leq k\\
c\neq 0 &r=k+1.\end{cases}$$ 
That is the order of contact is preserved. We can therefore write
$$\tilde{p}_{2}(x,\xi)=\xi_{1}-\xi_{2}^{k+1}g(x,\xi_{2})$$
for some $g(x,\xi_{2})$ with the property that $g(0,0)\neq 0$. Therefore
$$\tilde{p}_{2}(x,hD)=hD_{x_{1}}-(h^{k+1}D_{x_{1}}^{k+1})G(x,hD_{x_{2}})$$
where
$$G(x,\xi_{2})=g(x,\xi_{2})+O(h).$$
That is $G(x,hD_{x_{2}})$ is invertible. Now by Lemma \ref{lem:jqm} with $e_{1}(x,hD)=\Id$ and $e_{2}(x,hD)=-\Id$,  $v$ is a strong joint quasimode of order $h$  of $hD_{x_{1}}$ and $h^{k+1}D_{x_{2}}^{k+1}G(x,hD_{x_{2}}).$ A second application of Lemma \ref{lem:jqm} with $e_{1}(x,hD)=0$ and $e_{2}(x,hD)=G^{-1}(x,hD_{x_{2}})$ gives that $v$ is a strong joint quasimode of order $h$ of $hD_{x_{1}}$ and $h^{k+1}D_{x_{2}}^{k+1}$. 
\end{proof}

\section{Proof of Theorem \ref{thm:main}}
We are now in a position to prove Theorem \ref{thm:main}. We have that
$$u=W^{\star}(x_{1})v$$
where $v$ is a  strong joint quasimode of order $h$ of both $hD_{x_{1}}$ and $h^{k+1}D^{k+1}$. Therefore we expect the support of $\FT[v]$ to sit in a $h\times h^{\frac{1}{k+1}}$ about $\xi=0$. To incorporate the property  that $v$ is a $hD_{x_{1}}$ quasimode we decompose $v$ with a continuous wavelet transform in $x_{1}$. That is,
\begin{equation}v(x_{1},x_{2})=\frac{1}{C_{f}}\int \frac{1}{|a|^{5/2}}X_{v}(a,b,x_{2})f(a^{-1}(x_{1}-b))dadb\label{vint}\end{equation}
where
\begin{equation}X_{v}(a,b,x_{2})=\frac{1}{|a|^{1/2}}\int f(a^{-1}(y_{1}-b))v(y_{1},x_{2})dy_{1}\label{Xvdef}\end{equation}
and $f$ is a wavelet satisfying
$$\int f(\tau)d\tau=0\quad\text{and}\quad C_{f}=\int \frac{|\hat{f}(\xi)|^{2}}{|\xi|}d\xi<\infty.$$
Since $v$ is an order $h$ quasimode of $hD_{x_{1}}$ we would expect the major contributions to \eqref{vint} to come from the region where $|a|\sim 1$, and indeed (in Theorem \ref{thm:Wint}) we find this is the case. 

When we incorporate this decomposition we obtain
\begin{align}
W^{\star}(x_{1})v&=\frac{1}{C_{f}2\pi h}\int e^{\frac{i}{h}\left( \phi(x_{1},x_{2},\xi_{2})-y_{2}\xi_{2}\right)}b(x_{1},y_{2},\xi_{2})\frac{1}{|a|^{5/2}}X_{v}(a,b,x_{2})f(a^{-1}(x_{1}-b))dadbd\xi_{2}dy_{2}\label{decompW}\\
&=\frac{1}{C_{f}2\pi h}\int e^{\frac{i}{h}\phi(x_{1},x_{2},\xi_{2})}b(x_{1},y_{2},\xi_{2})f(a^{-1}(x_{1}-b))\FT\left[X_{v}\right](a,b,\xi_{2})d\xi_{2}dadb\nonumber\end{align}
where 
\begin{equation}\FT\left[X_{v}\right](a,b,\xi_{2})=\FT\left[X(a,b,\cdot)\right](\xi_{2})=\frac{1}{(2\pi h)^{\frac{1}{2}}}\int e^{-\frac{i}{h}x_{2}\xi_{2}}X_{v}(a,b,x_{2})dx_{2}.\label{FTXvdef}\end{equation}
We re-write \eqref{decompW} as
$$W^{\star}(x_{1})v=\frac{1}{C_f}\int \frac{1}{|a|^{5/2}}W_{a}^{\star}(x_{1})\left[\FT[X_{v}](a,\cdot,\cdot)\right]da$$
where
$$W_{a}^{\star}(x_{1})G=\frac{1}{(2\pi h)^{\frac{1}{2}}}\int e^{\frac{i}{h}\phi(x_{1},x_{2},\xi_{2})}b(x_{1},y_{x},\xi_{2})f(a^{-1}(x_{1}-b))G(\xi_{2},b)d\xi_{2}db.$$
Now we want to include the property that $v$ is also a quasimode of $h^{k+1}D^{k+1}$. Since we expect this restriction to require that $\FT[v]$ is supported mainly in the region $|\xi_{2}|\leq h^{\frac{1}{k+1}}$ we decompose $W_{a}^{\star}$ dyadically into regions with $|\xi_{2}|\sim 2^{j}h^{\frac{1}{k+1}}$. To that end choose $\chi_{0}:\R\to\R$ supported in $[-2,2]$  and $\chi$ supported in $[\frac{1}{2},\frac{3}{2}]$ so that
$$1=\chi_{0}(h^{-\frac{1}{k+1}}|\xi_{2}|)+\sum_{j=1}^{J}\chi(2^{-j}h^{-\frac{1}{k+1}}|\xi_{2}|)\quad\text{where}\quad 2^{J}h^{\frac{1}{k+1}}=1.$$
Then let
\begin{equation}W_{a,j}^{\star}(x_{1})G=\frac{1}{(2\pi h)^{\frac{1}{2}}}\int e^{\frac{i}{h}\phi(x_{1},x_{2},\xi_{2})}b(x_{1},y_{x},\xi_{2})f\left(\frac{x_{1}-b}{a}\right)\chi\left(\frac{|\xi_{2}|}{2^{j}h^{\frac{1}{k+1}}}\right)G(\xi_{2},b)d\xi_{2}db\label{wajdef}\end{equation}
\begin{equation}
W_{a,0}^{\star}(x_{1})G=\frac{1}{(2\pi h)^{\frac{1}{2}}}\int e^{\frac{i}{h}\phi(x_{1},x_{2},\xi_{2})}b(x_{1},x_{2},\xi_{2})f\left(\frac{x_{1}-b}{a}\right)\chi_{0}\left(\frac{|\xi_{2}|}{h^{k+1}}\right)G(\xi_{2},b)d\xi_{2}db.\label{wa0def}\end{equation}
We will proceed by proving $L^{2}\to L^{p}$ estimates for each $W_{a,j}^{\star}$, (in Theorem \ref{thm:Waest}), and then $L^{2}$ bounds for 
$$\FT[X_{v}^{j}]=\chi(2^{-j}h^{-\frac{1}{1+k}})\FT[X_{v}^{j}],$$
(in Propositions \ref{prop:Xvest} and \ref{prop:largea}). Finally in Theorem \ref{thm:Wint} we combine this information to obtain the estimates on $W(x_{1})^{\star}v$ necessary to prove Theorem \ref{thm:main}.

\begin{thm}\label{thm:Waest}
Suppose $W_{a,j}^{\star}(x_{1})$ and $W_{a,0}^{\star}(x_{1})$ are given by \eqref{wajdef} and \eqref{wa0def}. Then for $a\leq 2^{-2j}h^{1-\frac{2}{k+1}}$,
\begin{equation}\norm{W_{a,j}(\cdot)G}_{L^{p}}\lesssim a^{\frac{1}{2}+\frac{1}{p}}h^{\left(-1+\frac{1}{k+1}\right)\left(\frac{1}{2}-\frac{1}{p}\right)}2^{j\left(\frac{1}{2}-\frac{1}{p}\right)} \norm{G}_{L^{2}}.\label{Wsajopbnd}\end{equation}
For $a\geq{}2^{-2j}h^{1-\frac{2}{k+1}}$,
\begin{equation}\norm{W_{a,j}(\cdot)G}_{L^{p}}\lesssim a^{\frac{1}{2}} 2^{\mu(p,j)}h^{-\delta(p,k)}\norm{G}_{L^{2}}\label{Wbajopbnd}\end{equation}
where 
\begin{align*}
\delta(p,k)&=\begin{cases}
\left(\frac{1}{2}-\frac{2}{p}\right)-\frac{1}{k+1}\left(\frac{1}{2}-\frac{3}{p}\right)&6\leq p\leq\infty\\
\frac{1}{4}-\frac{1}{2p}&2\leq p\leq 6,\end{cases}\\
&\text{and}\\
\mu(p,j)&=\begin{cases}
j\left(\frac{1}{2}-\frac{3}{p}\right)&6\leq{}p\leq \infty\\
0 & 2\leq p\leq 6.\end{cases}\end{align*}
\end{thm}

\begin{proof}
Here we use a well developed technique to obtain the $L^{p}$ estimates. Indeed the general idea dates back Tomas-Stein's \cite{Tom75} treatment of $L^{2}\to L^{p}$ extension (dual restriction) results. There are three important steps.
\begin{enumerate}
\item Thinking again as $x_{1}$ as the `time' variable we aim to compute a Strichartz $L^{p}_{x_{1}}L^{p}_{x_{2}}$ estimate. To that end we compute $W_{a,j}^{\star}(x_{1})W_{a,j}(z_{1})$ and find estimates of the form
\begin{align*}
\norm{W_{a,j}^{\star}(x_{1})W_{a,j}(z_{1})}_{L^{1}\to L^{\infty}}&\lesssim h^{-\gamma_{\infty}}Q_{\infty}(|x_{1}-z_{1}|,h,a)\\
\norm{W_{a,j}^{\star}(x_{1})W_{a,j}(z_{1})}_{L^{2}\to L^{2}}&\lesssim h^{-\gamma_{2}}Q_{2}(|x_{1}-z_{1}|,h,a)\end{align*}
where the $Q_{i}(|x_{1}-z_{1}|,h,a)$ captures decay as $|x_{1}-z_{1}|$ increases.
\item As in the Keel-Tao \cite{KT98} treatment of abstract Strichartz estimates we interpolate to obtain 
$$\norm{W_{a,j}^{\star}(x_{1})W_{a,j}(z_{1})}_{L^{p'}\to L^{p}}\lesssim h^{-\gamma_{p}}Q_{p}(|x_{1}-z_{1}|,h,a).$$
\item Finally we resolve the $|x_{1}-z_{1}|$ integral using Young's inequality, or Hardy-Littlewood-Sobolev's inequality in the borderline case where $Q_{p}(|x_{1}-z_{1}|,h,a)^{p/2}$ just fails to be integrable.
\end{enumerate}

First note that for any $x_{1}$
$$\norm{W_{a,j}(x_{1})}_{L^{2}_{x_{2}}\to L^{2}_{x_{2}}}\lesssim\norm{W_{a}(x_{1})}_{L^{2}_{x_{2}}\to L^{2}_{x_{2}}}\lesssim 1.$$
Then using the support properties of $f$ and the fact that $v$ localised in an $O(1)$ region we see that
$$\norm{W_{a,j}(x_{1})}_{L^{2}_{x_2}L^{2}_{b}\to L^{2}_{x_{1}}L^{2}_{x_{2}}}\lesssim a^{1/2}.$$
So
$$\norm{W_{a,j}^{\star}(x_{1})W_{a,j}(z_{1})}_{L^{2}\to L^{2}}\lesssim a.$$
Therefore we need only focus on the $L^{1}\to L^{\infty}$ estimate. Computing $W_{a,j}^{\star}W_{a,j}$ we have that,
$$W_{a,j}^{\star}(x_{1})W_{a,j}(z_{1})g=W_{a,j}^{\star}W_{a,j}(x_{1},z_{1})g=\int K(x_{1},x_{2},z_{1},z_{2})g(z_{1},z_{2})dz_{1}dz_{2}$$
where for $1\leq{}j\leq J$
\begin{multline*}
K_{j}(x_{1},x_{2},z_{1},z_{2})=\frac{1}{2\pi h}\int e^{\frac{i}{h}\left(\phi(x_{1},x_{2},\xi_{2})-\phi(z_{1},z_{2},\xi_{2})\right)}f(a^{-1}(x_{1}-b))f(a^{-1}(z_{1}-b))\\
\times \chi^{2}(2^{-j}h^{-\frac{1}{k+1}}|\xi_{2}|)B(x_{1},x_{2},z_{1},z_{2},\xi_{2})d\xi_{2}db\end{multline*}
and
\begin{multline*}
K_{0}(x_{1},x_{2},z_{1},z_{2})=\frac{1}{2\pi h}\int e^{\frac{i}{h}\left(\phi(x_{1},x_{2},\xi_{2})-\phi(z_{1},z_{2},\xi_{2})\right)}f(a^{-1}(x_{1}-b))f(a^{-1}(z_{1}-b))\\
\times\chi^{2}_{0}(h^{-\frac{1}{k+1}}\xi_{2})B(x_{1},x_{2},z_{1},z_{2},\xi_{2})d\xi_{2}db.\end{multline*}
Note that for the kernel to be non-zero we require that $|x_{1}-z_{1}|\leq{}Ca$ for some sufficiently large constant. Also since $v$ is localised to an $O(1)$ region we must  have $|x_{1}-z_{1}|\leq 1$. We estimate the $b$ integral using the support properties of $f$. To compute the $\xi_{2}$ integral we want to appeal to the stationary phase lemma (it is here that we use the curvature assumption on $p_{1}(x,\xi)$). As it \cite{koch} and \cite{tacy09} we have that
\begin{multline*}
\phi(x_{1},x_{2},\xi_{2})-\phi(z_{1},z_{2},\xi_{2})=\langle x_{2}-z_{2},\xi_{2}+z_{1}F(z_1,x_{2},z_{2},\xi_{2})\rangle\\
+(x_{1}-z_{1})a(0,x_{2},\xi_{2})+O(|x_{1}-z_{1}|^{2}).\end{multline*}
The curvature assumption guarantees that $|\partial^{2}_{\xi_{2}\xi_{2}}a(0,x_{2},\xi_{2})|>c>0$ so
$$|\partial^{2}_{\xi_{2}\xi_{2}}(\phi(x_{1},x_{2},\xi_{2})-\phi(z_{1},z_{2},\xi_{2}))|\geq{}c|x_{1}-z_{1}|.$$
Therefore, had the symbol had been smooth in $\xi_{2}$, the stationary phase lemma would tell us that greatest contribution would come from a $h^{1/2}|x_{1}-z_{1}|^{-1/2}$ region about the critical point. However, the cut off in $\xi_{2}$ is not smooth. In fact for $|x_{1}-z_{1}|\leq 2^{-2j}h^{1-\frac{2}{k+1}}$ we get a better estimate by just using the support properties of $\chi$ (or $\chi_{0}$). When $|x_{1}-z_{1}|\geq{}2^{-2j}h^{1-\frac{2}{k+1}}$ the regularity of the cut off is equal or better than the natural regularity introduced in the proof of the stationary phase lemma. Therefore
$$|K_{j}(x_{1},x_{2},z_{1},z_{2})|\leq{}
\begin{cases}
a2^{j}h^{-1+\frac{1}{k+1}}&|x_{1}-z_{1}|\leq{}2^{-2j}h^{1-\frac{2}{k+1}}\\
ah^{-\frac{1}{2}}|x_{1}-z_{1}|^{-1/2}&|x_{1}-z_{1}|\geq{}2^{-2j}h^{1-\frac{2}{k+1}}\end{cases}$$
and
$$\norm{W_{a,j}^{\star}W_{a,j}(x_{1},z_{1})}_{L^{1}\to L^{\infty}}\lesssim \begin{cases}
a2^{j}h^{-1+\frac{1}{k+1}}&|x_{1}-z_{1}|\leq{}2^{-2j}h^{1-\frac{2}{k+1}}\\
ah^{-\frac{1}{2}}|x_{1}-z_{1}|^{-1/2}&|x_{1}-z_{1}|\geq{}2^{-2j}h^{1-\frac{2}{k+1}}.\end{cases}$$
Since we have the $L^{2}\to L^{2}$ bound of
$$\norm{W_{a,j}^{\star}W_{a,j}(x_{1},z_{1})}_{L^{2}\to L^{2}}\lesssim a$$
we can interpolate to get
$$\norm{W_{a,j}^{\star}W_{a,j}(x_{1},z_{1})}_{L^{p'}\to L^{p}}\lesssim\begin{cases}
a2^{2j\left(\frac{1}{2}-\frac{1}{p}\right)}h^{2\left(-1+\frac{1}{k+1}\right)\left(\frac{1}{2}-\frac{1}{p}\right)}&|x_{1}-z_{1}|\leq 2^{-2j}h^{1-\frac{2}{k+1}}\\
ah^{-\left(\frac{1}{2}-\frac{1}{p}\right)}|x_{1}-z_{1}|^{-\left(\frac{1}{2}-\frac{1}{p}\right)}&|x_{1}-z_{2}|\geq 2^{-2j} h^{1-\frac{2}{k+1}}.\end{cases}$$
We are now in a position to examine our two cases
\begin{itemize}
\item[Case 1] $a\leq 2^{-2j}h^{1-\frac{2}{k+1}}$
\item[Case 2] $a \geq{}2^{-2j}h^{1-\frac{2}{k+1}}.$
\end{itemize}

In Case 1 the restriction that $|x_{1}-z_{1}|\leq Ca$ means that we are always in the situation where
$$\norm{W_{a,j}^{\star}W_{a,j}(x_{1},z_{1})}_{L^{p'}\to L^{p}}\lesssim a2^{2j\left(\frac{1}{2}-\frac{1}{p}\right)}h^{2\left(-1+\frac{1}{k+1}\right)\left(\frac{1}{2}-\frac{1}{p}\right)}.$$
So
\begin{align*}
\norm{W_{a,j}^{\star}W_{a,j}(\cdot,\cdot)}_{L^{p'}\to L^{p}}&\lesssim a2^{2j\left(\frac{1}{2}-\frac{1}{p}\right)}h^{2\left(-1+\frac{1}{k+1}\right)\left(\frac{1}{2}-\frac{1}{p}\right)}\left(\int_{0}^{a}dt\right)^{\frac{2}{p}}\\
&\lesssim a^{1+\frac{2}{p}}2^{2j\left(\frac{1}{2}-\frac{1}{p}\right)}h^{2\left(-1+\frac{1}{k+1}\right)\left(\frac{1}{2}-\frac{1}{p}\right)}.\end{align*}
In Case 2 (for $p\neq 6$) Young's inequality gives
\begin{multline*}
\norm{W_{a,j}^{\star}W_{a,j}(\cdot,\cdot)}_{L^{p'}\to L^{p}}\lesssim a2^{2j\left(\frac{1}{2}-\frac{1}{p}\right)}h^{2\left(-1+\frac{1}{k+1}\right)\left(\frac{1}{2}-\frac{1}{p}\right)}\left(\int_{0}^{2^{-2j}h^{1-\frac{2}{k+1}}}dt\right)^{\frac{2}{p}}\\
+ah^{-\left(\frac{1}{2}-\frac{1}{p}\right)}\left(\int_{2^{-2j}h^{1-\frac{2}{k+1}}}^{a}|t|^{-\frac{p}{2}\left(\frac{1}{2}-\frac{1}{p}\right)}dt\right)^{\frac{2}{p}}\end{multline*}
and so
$$\norm{W_{a,j}^{\star}W_{a,j}(\cdot,\cdot)}_{L^{p'}\to L^{p}}\lesssim 2^{2\mu(p,j)}h^{-2\delta(p,j)}.$$
Using Hardy-Littlewood-Sobolev's inequality to resolve the the case where $p=6$ we obtain
$$\norm{W_{a,j}^{\star}W_{a,j}(\cdot,\cdot)}_{L^{6/5}\to L^{6}}\lesssim 2^{2\mu(6,j)}h^{-2\delta(6,j)}.$$

\end{proof}

Finally we can put this information together with the estimates from Proposition \ref{prop:Xvest} and Proposition \ref{prop:largea} (whose proofs we defer to the end of this section). 

\begin{thm}\label{thm:Wint}
If $W(x_{1})$ and $v$ are given by \eqref{Wdef} and \eqref{vdef} respectively,
\begin{equation}\norm{W^{\star}(\cdot)v}_{L^{p}}\lesssim h^{-\delta(p,k)}\norm{u}_{L^{2}}.\label{fullest}\end{equation}
\end{thm}

\begin{proof}
Clearly
$$\norm{W^{\star}v}_{L^{p}}\lesssim \sum_{j=0}^{J}\norm{W^{\star}_{j}v}_{L^{p}}$$
where
$$W^{\star}_{j}v=\frac{1}{C_f}\int \frac{1}{|a|^{5/2}}W_{a,j}^{\star}\left[\FT[X_{v}](a,b,\cdot)\right]da.$$
We divide the $a$ integral into three parts
$$W^{\star}_{j}v=T_{j,1}[\FT[X_{v}]]+T_{j,2}[\FT[X_{v}]]+T_{j,3}[\FT[X_{v}]]$$
with
\begin{align}
T_{j,1}[\FT[X_{v}]]&=\frac{1}{C_{f}}\int_{|a|\leq 2^{-2j}h^{1-\frac{2}{k+1}}} \frac{1}{|a|^{5/2}}W_{a,j}\left[\FT[X_{v}](a,\cdot,\cdot)\right]da,\label{T1def}\\
T_{j,2}[\FT[X_{v}]]&=\frac{1}{C_{f}}\int_{2^{-2j}h^{1-\frac{2}{k+1}}\leq{}|a|\leq 1}\frac{1}{|a|^{5/2}}W_{a,j}\left[\FT[X_{v}](a,\cdot,\cdot)\right]da,\label{T2def}\\
&\text{and}\nonumber\\
T_{j,3}[\FT[X_{v}]]&=\frac{1}{C_{f}}\int_{|a|\geq{}1}\frac{1}{|a|^{5/2}}W_{a,j}\left[\FT[X_{v}](a,\cdot,\cdot)\right]da.\label{T3def}\end{align}
Note that the localisation  applied to $W_{a,j}^{\star}$ means that it is enough to consider $T_{j,i}$ acting on $\FT[X_{v}^{j}]$ given by
$$\begin{cases}\FT[X_{v}^{j}](a,b,\xi_{2})=\chi\left(2^{-j}h^{-\frac{1}{k+1}}|\xi_{2}|\right)\FT[X_{v}](a,b,\xi_{2})&j\geq 1\\
\FT[X_{v}^{0}](a,b,\xi_{2})=\chi_{0}\left(h^{-\frac{1}{k+1}}\xi_{2}\right)\FT[X_{v}](a,b,\xi_{2})&j=0.\end{cases}$$
Let's treat $T_{j,1}$ first. Notice that since $k\geq{}1$, $2^{-2j}h^{1-\frac{2}{k+1}}\leq{}1$ so we can use the result of Proposition \ref{prop:Xvest}, namely for $|a|\leq{}1$,
$$\norm{\FT[X_{v}^{j}](a,\cdot,\cdot)}_{L^{2}_{\xi_{2}}L^{2}_{b}}\lesssim a^{\frac{3}{2}}2^{-jM}$$ 
for any natural number $M$. 
Then
\begin{align*}
\norm{T_{j,1}[\FT[X^{j}_{v}]]}_{L^{p}}&\lesssim \int_{|a|\leq{}2^{-2j}h^{1-\frac{2}{k+1}}}|a|^{-2+\frac{1}{p}}h^{\left(-1+\frac{1}{k+1}\right)\left(\frac{1}{2}-\frac{1}{p}\right)}2^{j\left(\frac{1}{2}-\frac{1}{p}\right)}\norm{\FT[X^{j}_{v}](a,\cdot,\cdot)}_{L^{2}_{\xi_{2}}L^{2}_{b}}\\
&\lesssim \int_{|a|\leq 2^{-2j}h^{1-\frac{2}{k+1}}}|a|^{-\frac{1}{2}+\frac{1}{p}}h^{\left(-1+\frac{1}{k+1}\right)\left(\frac{1}{2}-\frac{1}{p}\right)}2^{j\left(\frac{1}{2}-\frac{1}{p}-M\right)}\norm{u}_{L^{2}}\\
&\lesssim h^{-\frac{1}{2(k+1)}+\frac{1}{p}\left(2-\frac{3}{k+1}\right)}2^{j\left(-\frac{1}{2}-\frac{2}{p}-M\right)}\norm{u}_{L^{2}}\end{align*}
which is better than \eqref{fullest}. Now considering $T_{j,2}$ we have
\begin{align*}
\norm{T_{j,2}[\FT[X^{j}_{v}]]}_{L^{p}}&\lesssim \int_{2^{-2j}h^{1-\frac{2}{k+1}}\leq|a|\leq 1}a^{-2} 2^{\mu(p,j)}h^{-\delta(p,k)}\norm{\FT[X^{j}_{v}](a,\cdot,\cdot)}_{L^{2}_{\xi_{2}}L^{2}_{b}}da\\
&\lesssim \int_{2^{-2j}h^{1-\frac{2}{k+1}}\leq|a|\leq 1}a^{-\frac{1}{2}} 2^{\mu(p,j)-jM}h^{-\delta(p,k)}\norm{u}_{L^{2}}\\
&\lesssim 2^{\mu(p,j)-Mj}h^{-\delta(p,k)}\norm{u}_{L^{2}}.\end{align*}
Finally to deal with $T_{j,3}$ we use the result of Proposition \ref{prop:largea} which states that for $|a|\geq{}1$ 
$$\norm{\FT[X_{v}^{j}](a,\cdot,\cdot)}_{L^{2}_{\xi_{2}}L^{2}_{b}}\lesssim 2^{-jM}$$ 
for any natural number $M$. This gives us
\begin{align*}
\norm{T_{j,3}[\FT[X^{j}_{v}]]}_{L^{p}}&\lesssim\int_{1\leq{}|a|}a^{-2}2^{\mu(p,j)}h^{-\delta(p,k)}\norm{\FT[X^{j}_{v}](a,\cdot,\cdot)}_{L^{2}_{\xi_{2}}L^{2}_{b}}\\
&\lesssim \int_{1\leq|a|}a^{-2}2^{\mu(p,j)-jM}h^{-\delta(p,k)}\norm{u}_{L^{2}}\\
&\lesssim 2^{\mu(p,j)-jM}h^{-\delta(p,k)}\norm{u}_{L^{2}}.\end{align*}
Therefore by making $M$ large enough
$$\norm{W^{\star}_{j}(\cdot)v}_{L^{p}}\lesssim h^{-\delta(p,k)}2^{-2j}\norm{u}_{L^{2}}$$
and so summing
$$\norm{W^{\star}(\cdot)v}_{L^{p}}\lesssim h^{-\delta(p,k)}\norm{u}_{L^{2}}.$$

\end{proof}

Below we provide the proofs of Propositions \ref{prop:Xvest} and \ref{prop:largea}.

\begin{prop}\label{prop:Xvest}
Suppose $v$ is a  strong  joint quasimode of order $h\norm{u}_{L^{2}}$ of $hD_{x_{1}}$ and $h^{k+1}D^{k+1}_{x_{2}}$ and  $\FT\left[X_{v}^{j}\right](a,b,\xi_{2})$
is given by
$$\FT\left[X_{v}^{j}\right](a,b,\xi_{2})=\chi\left(2^{-j}h^{-\frac{1}{k+1}}\xi_{2}\right)\FT\left[X_{v}(a,b,\cdot)\right](\xi_{2})\quad 1\leq j\leq J,$$
$$\FT\left[X_{v}^{0}\right](a,b,\xi_{2})=\chi_{0}\left(h^{-\frac{1}{k+1}}\xi_{2}\right)\FT\left[X_{v}(a,b,\cdot)\right](\xi_{2}).$$
 Then for all $|a|\leq{}1$
\begin{equation}\norm{\FT\left[X_{v}^{j}\right](a,\cdot,\cdot)}_{L^{2}_{b}L^{2}_{\xi_{2}}}\lesssim 2^{-jM}a^{\frac{3}{2}}\norm{u}_{L^{2}}\label{Xvest}\end{equation}
for any  $M\in\N$.
\end{prop}

\begin{proof}
The semiclassical Fourier transform preserves $L^{2}$ norms so for fixed $a$ and $b$
$$\norm{\FT\left[X_{v}^{0}\right](a,b,\cdot)}_{L^{2}_{\xi_{2}}}\lesssim\norm{\FT\left[X_{v}\right](a,b,\cdot)}_{L^{2}_{\xi_{2}}}=\norm{X_{v}(a,b,\cdot)}_{L^{2}_{x_{2}}}.$$
Now
$$X_{v}(a,b,x_{2})=\frac{1}{|a|^{1/2}}\langle f(a^{-1}(x_{1}-b),v(x_{1},x_{2})\rangle_{x_{1}}.$$
Since $f$ has integral zero it can be written as the derivative of a function $g$ which is also compactly supported. That is
$$f(x_{1})=iD_{x_{1}}g(x_{1})\Rightarrow f(a^{-1}(x_{1}-b))=aiD_{x_{1}}g(a^{-1}(x_{1}-b)).$$
Therefore
\begin{align*}
X_{v}(a,b,x_{2})&=\frac{ia}{|a|^{1/2}}\langle D_{x_{1}}g(a^{-1}(x_{1}-b)),v(x_{1},x_{2})\rangle_{x_{1}}\\
&=-\frac{ia}{|a|^{1/2}}\langle g(a^{-1}(x_{1}-b),D_{x_{1}}v(x_{1},x_{2})\rangle_{x_{1}}.\end{align*}
So
$$\left|X_{v}(a,b,x_{2})\right|\lesssim a\norm{g}_{L^{2}}\norm{D_{x_{1}}v(x_{1},x_{2})}_{L^{2}_{x_{1}}(B_{a}(b))}$$
and
$$\norm{X_{v}(a,b,\cdot)}_{L^{2}_{x_{2}}}\lesssim a\norm{D_{x_{1}}v(x_{1},x_{2})}_{L^{2}_{x_{1}}(B_{a}(b))L^{2}_{x_{2}}}.$$
Therefore
$$\norm{\FT\left[X_{v}^{0}\right](a,b,\cdot)}_{L^{2}_{\xi_{2}}}\lesssim a\norm{D_{x_{1}}v(x_{1},x_{2})}_{L^{2}_{x_{1}}(B_{a}(b))L^{2}_{x_{2}}}.$$
Computing the $b$ integral we have
\begin{align*}
\norm{\FT\left[X_{v}^{0}\right](a,\cdot,\cdot)}_{L^{2}_{b}L^{2}_{\xi_{2}}}&\lesssim a^{\frac{3}{2}}\norm{D_{x_{1}}v}_{L^{2}}\\
&\lesssim a^{\frac{3}{2}}\norm{u}_{L^{2}}.\end{align*}
Where in the final step we have used the fact that $v$ is an $O_{L^{2}}(h\norm{u}_{L^{2}})$ quasimode of $hD_{x_{1}}$.

When $j\geq{}1$, we first write
$$\FT\left[X_{v}^{j}(a,b,\cdot)\right](\xi_{2})=\frac{1}{\xi_{2}^{(k+1)M}}\xi_{2}^{(k+1)M}\FT\left[X_{v}^{j}(a,b,\cdot)\right](\xi_{2})$$
and see that
\begin{align*}
\norm{\FT\left[X_{v}^{j}\right](a,b,\cdot)}_{L^{2}_{\xi_{2}}}&\lesssim h^{-M}2^{-j(k+1)M}\norm{\xi_{2}^{(k+1)M}\FT\left[X_{v}\right](a,b,\cdot)}_{L^{2}_{\xi_{2}}}\\
&=h^{-M}2^{-j(k+1)M}\norm{(h^{k+1}D_{x_{2}}^{k+1})^{M}X_{v}(a,b,\cdot)}_{L^{2}_{x_{2}}}.\end{align*}
Now
$$(h^{k+1}D^{k+1}_{x_{2}})^{M}X_{v}^{j}(a,b,x_{2})=(h^{k+1}D^{k+1}_{x_{2}})^{M}\frac{1}{|a|^{1/2}}\langle f(a^{-1}(x_{1}-b),v(x_{1},x_{2})\rangle_{x_{1}}.$$
Following the same manipulations as the $j=0$ case
\begin{align*}
(h^{k+1}D^{k+1}_{x_{2}})^{M}X_{v}^{j}(a,b,x_{2})&=-\frac{ia}{|a|^{1/2}}\langle g(a^{-1}(x_{1}-b),(h^{k+1}D^{k+1}_{x_{2}})^{M}D_{x_{1}}v(x_{1},x_{2})\rangle_{x_{1}},\\
\left|(h^{k+1}D^{k+1}_{x_{2}})^{M}X_{v}^{j}(a,b,x_{2})\right|&\lesssim a \norm{(h^{k+1}D^{k+1}_{x_{2}})^{M}D_{x_{1}}v(x_{1},x_{2})}_{L^{2}_{x_{1}}(B_{a}(b))}\end{align*}
and
$$\norm{(h^{k+1}D^{k+1}_{x_{2}})^{M}X_{v}^{j}(a,b,\cdot)}_{L^{2}_{x_{2}}}\lesssim a \norm{(h^{k+1}D^{k+1}_{x_{2}})^{M}D_{x_{1}}v}_{L^{2}_{x_{1}}(B_{a}(b))L^{2}_{x_{2}}}.$$
Therefore
$$\norm{\xi_{2}^{(k+1)M}\FT\left[X_{v}^{j}\right](a,b,\cdot)}_{L^{2}_{\xi_{2}}}\lesssim a\norm{(h^{k+1}D^{k+1}_{x_{2}})^{M}D_{x_{1}}v}_{L^{2}_{x_{1}}(B_{a}(b))L^{2}_{x_{2}}}.$$
and so
$$\norm{\FT\left[X_{v}^{j}\right](a,b,\cdot)}_{L^{2}_{\xi_{2}}}\lesssim ah^{-M}2^{j(k+1)M}\norm{(h^{k+1}D^{k+1}_{x_{2}})^{M}D_{x_{1}}v}_{L^{2}_{x_{1}}(B_{a}(b))L^{2}_{x_{2}}}.$$
Integrating in $b$ gives us
$$\norm{\FT\left[X_{v}^{j}\right](a,\cdot,\cdot)}_{L^{2}_{\xi_{2}}L^{2}_{b}}\lesssim a^{\frac{3}{2}}\norm{(h^{k+1}D^{k+1}_{x_{2}})^{M}D_{x_{1}}v}_{L^{2}}.$$
Finally since $v$ is a strong joint quasimode of order $h$ of $hD_{x_{1}}$ and $h^{k+1}D_{x_{2}}^{k+1}$ we obtain
$$\norm{\FT\left[X_{v}^{j}\right](a,\cdot,\cdot)}_{L^{2}_{\xi_{2}}L^{2}_{b}}\lesssim a^{\frac{3}{2}}2^{-j(k+1)M}\norm{u}_{L^{2}}.$$ 

\end{proof}

\begin{prop}\label{prop:largea}
Suppose $v$ is a strong joint quasimode of order $h$ of $hD_{x_{1}}$ and $h^{k+1}D^{k+1}_{x_{2}}$ and $\FT\left[X_{v}^{j}\right]$ is as in Proposition \ref{prop:Xvest}. Then for $|a|\geq{}1$,
$$\norm{\FT[X_{v}](a,\cdot,\cdot)}_{L^{2}_{\xi_{2}}L^{2}_{b}}\lesssim 2^{-jM}\norm{u}_{L^{2}}$$
for any  $M\in\N$.
\end{prop}

\begin{proof}
For $j=0$ we again use the fact that the semiclassical Fourier transform preserves $L^{2}$ norms to write,
$$\norm{\FT\left[X_{v}^{0}\right](a,b,\cdot)}_{L^{2}_{\xi_{2}}}\lesssim\norm{\FT\left[X_{v}\right](a,b,\cdot)}_{L^{2}_{\xi_{2}}}=\norm{X_{v}(a,b,\cdot)}_{L^{2}_{x_{2}}}.$$
Referring to the definition, \eqref{Xvdef}, of $X_{v}(a,b,x_{2})$ we have
$$X_{v}(a,b,x_{2})=\frac{1}{|a|^{1/2}}\langle f(a^{-1}(x_{1}-b)),v(x_{1},x_{2})\rangle_{x_{1}}.$$
Remember that $v(x_{1},x_{2})$ is localised to an $O(1)$ region. Therefore
$$|X_{v}(a,b,x_{2})|\lesssim \frac{1}{|a|^{1/2}}\norm{v}_{L^{2}_{x_{1}}}$$
and
$$\norm{\FT\left[X_{v}\right]}_{L^{2}_{\xi_{2}}}=\norm{X_{v}(a,b,\cdot)}_{L^{2}_{x_{2}}}\lesssim \frac{1}{|a|^{1/2}}\norm{v}_{L^{2}}.$$
Now once $|b|\gg |a|$ the support of $f(a^{-1}(x_{1}-b))$ no longer overlaps with the support of $v$ so
$$\norm{\FT\left[X_{v}\right](a,\cdot,\cdot)}_{L^{2}_{\xi_{2}}L^{2}_{b}}\lesssim\norm{v}_{L^2}\lesssim\norm{u}_{L^{2}}.$$
If $j\geq 1$, $\xi_{2}$ is supported away from zero on the support of $\FT\left[X_{v}^{j}\right]$. So again writing
$$\FT\left[X_{v}^{j}(a,b,\cdot)\right](\xi_{2})=\frac{1}{\xi_{2}^{(k+1)M}}\xi_{2}^{(k+1)M}\FT\left[X_{v}^{j}(a,b,\cdot)\right](\xi_{2})$$
we see that
\begin{align*}
\norm{\FT\left[X_{v}^{j}\right](a,b,\cdot)}_{L^{2}_{\xi_{2}}}&\lesssim h^{-M}2^{-j(k+1)M}\norm{\xi_{2}^{(k+1)M}\FT\left[X_{v}\right](a,b,\cdot)}_{L^{2}_{\xi_{2}}}\\
&=h^{-M}2^{-j(k+1)M}\norm{(h^{k+1}D_{x_{2}}^{k+1})^{M}X_{v}(a,b,\cdot)}_{L^{2}_{x_{2}}}.\end{align*}
As in the $j=0$ case we write
$$X_v(a,b,x_{2})=\frac{1}{|a|^{1/2}}\langle f(a^{-1}(x_{1}-b)),v(x_{1},x_{2})\rangle_{x_{1}}$$
so
\begin{align*}
|(h^{k+1}D^{k+1}_{x_{2}})^{M}X_{v}(a,b,x_{2})|&\lesssim\frac{1}{|a|^{1/2}}\left|\langle f(a^{-1}(x_{1}-b)),(h^{k+1}D^{k+1}_{x_{2}})^{M}v(x_{1},x_{2})\rangle_{x_{1}}\right|\\
&\lesssim \frac{1}{|a|^{1/2}}\norm{(h^{k+1}D^{k+1})^{M}v}_{L^{2}_{x_{1}}}.\end{align*}
Since $v$ is a strong $O_{L^{2}}(h\norm{u}_{L^{2}})$ quasimode of $h^{k+1}D_{x_{2}}^{k+1}$
$$\norm{(h^{k+1}D^{k+1}_{x_{2}})^{M}X_{v}(a,b,\cdot)}_{L^{2}_{x_{2}}}\lesssim \frac{h^{M}}{|a|^{1/2}}\norm{u}_{L^{2}}$$
and
$$\norm{\FT\left[X_{v}^{j}\right](a,b,\cdot)}_{L^{2}_{\xi_{2}}}\lesssim \frac{2^{-j(k+1)M}}{|a|^{1/2}}\norm{u}_{L^{2}}.$$
Since $X_{v}(a,b,x_{2})=0$ for $b\gg a$
$$\norm{\FT\left[X_{v}^{j}\right](a,\cdot,\cdot)}_{L^{2}_{\xi_{2}}}\lesssim 2^{-j(k+1)M}\norm{u}_{L^{2}}.$$
\end{proof}

\section*{Appendix A: Semiclassical analysis}\label{sec:semiclassical}

Throughout this paper we have used  some  standard results from semiclassical analysis. For the readers convenience we record the results in this appendix  and direct them to \cite{Zworski12} and \cite{Mart02} for the proofs and further discussion. In this paper we always use the left quantisation of semiclassical pseudodifferential operators. That is given a symbol $p(x,\xi)$ we define the operator
$$Op_{h}(p)u=p(x,hD)u=\frac{1}{(2\pi h)^{n}}\int e^{\frac{i}{h}\langle x-y,\xi\rangle}p(x,\xi)u(y)d\xi dy$$

\begin{prop}[Composition of semiclassical pseudodifferential operators]\label{prop:com}
Let $p(x,hD)$, $q(x,hD)$ be left-quantised semiclassical pseudodifferential operators with symbols $p(x,\xi)$ and $q(x,\xi)$ respectively. The the symbol of $p(x,hD)\circ{}q(x,hD)$ is given by
\begin{multline}p(x,\xi)\# q(x,\xi)=e^{ih\langle{}D_{\xi},D_{y}\rangle}p(x,\xi)q(y,\eta)\Big|_{x=y,\xi=\eta}\\
  =\sum_{k}\frac{h^{k}}{k!}\left(\frac{\langle{}D_{\xi},D_{y}\rangle}{i}\right)^{k}p(x,\xi)q(y,\eta)\Big|_{x=y,\xi=\eta.}\label{semiexp}\end{multline}
  \end{prop}
  
  \begin{prop}[Commutation identity]\label{prop:commute}
  Let $p(x,hD)$, $q(x,hD)$ be left-quantised semiclassical pseudodifferential operators. Then
  $$[p(x,hD),q(x,hD)]=hOp_{h}(\{p,q\})+h^{2}r(x,hD)$$
  where $\{p,q\}$ is the Poisson bracket of $p$ and $q$ and the remainder operator $r(x,hD)$ satifies
$$\norm{r(x,hD)}_{L^{2}\to{}L^{2}}\lesssim 1.$$
\end{prop}

\begin{prop}[Invertibility of elliptic operators]\label{prop:invert}
Let $p(x,hD)$ be a left-quantised, semiclassical pseudodifferential operator with symbol $p(x,\xi)$ such that $|p(x,\xi)|>c>0$. Then there exists an inverse operator $(p(x,hD))^{-1}$ with
$$\norm{(p(x,hD))^{-1}}_{L^{2}\to L^{2}}\lesssim 1.$$
\end{prop}

\begin{prop}[Parametrix representation of $W(x_{1})$]\label{prop:parametrix}
Suppose $W(x_{1})$ has the property that
$$hD_{x_{1}}W(x_{1})=-W(x_{1})a(x,hD_{x_{2}})+O_{L^{2}}(h^{\infty})$$
then $W(x_{1})$ has the parametrix representation 
\begin{equation}W(x_{1})g=\frac{1}{2\pi h}\int e^{\frac{i}{h}\left(x_{2}\xi_{2}-\phi(x_{1},y_{2},\xi_{2})\right)}b(x_{1},y_{2},\xi_{2})g(y_{2})d\xi_{2}dy_{2}\label{Wpara}\end{equation}
where
\begin{equation}\partial_{x_{1}}\phi+a(x,\partial_{y_{2}}\phi,\xi_{2})=0\quad \phi(0,y_{2},\xi_{2})=y_{2}\xi_{2}\label{eikeq}\end{equation}
$$b(0,y_{2},\xi_{2})=1.$$
\end{prop}

\begin{proof}
We provide only a sketch of this standard result. Using the parametrix \eqref{Wpara} we have that
$$hD_{x_{1}}W(x_{1})g=\frac{1}{2\pi h}\int e^{\frac{i}{h}\left(x_{2}\xi_{2}-\phi(x_{1},y_{2},\xi_{2})\right)}\left(\partial_{x_{1}}\phi b+h\partial_{x_{1}}b\right)(x_{1},y_{2},\xi_{2})g(y_{2})d\xi_{2}dy_{2}$$
and
$$-Wa(x,hD_{x_{2}})g=-\frac{1}{(2\pi h)}\int e^{\frac{i}{h}\left(x_{2}\xi_{2}-\phi(x_{1},y_{2},\xi_{2})+(y_{2}-z_{2})\eta_{2}\right)}b(x_{1},y_{2},\xi_{2})a(x_{1},y_{2},\eta_{2})g(z_{2})d\xi_{2}dy_{2}d\eta_{2}dz.$$
Now computing the $(y_{2},\eta_{2})$ integral via the method of stationary phase we find that there is a non-degenerate critical point where
$$y_{2}-z_{2}=0\quad \text{and}\quad -\partial_{y_2}\phi+\eta_{2}=0.$$
So
$$-Wa(x,hD_{x_{2}})g=-\frac{1}{2\pi h}\int e^{\frac{i}{h}\left(x_{2}\xi_{2}-\phi(x_{2},z_{2},\xi_{2})\right)}\left(b(x_{1},z_{2},\xi_{2})a(x_{1},z_{2},\partial_{y_{2}}\phi)+hr(x_{1},z_{2},\xi_{2})\right)g(z_{2})d\xi_{2}dz$$
where $r(x_{1},z_{2},\xi_{2})$ is determined by the lower order terms in the stationary phase expansion. Therefore if $\phi$ satisfies \eqref{eikeq},
$$\left(hD_{x_{1}}W(x_{1})+W(x_{1})a(x,hD_{x_{2}})\right)g=\frac{1}{2\pi h}\int e^{\frac{i}{h}\left(x_{2}\xi_{2}-\phi(x_{2},z_{2},\xi_{2})\right)}h\left(\partial_{x_{1}}b+r\right)(x_{1},z_{2},\xi_{2})g(z_{2})d\xi_{2}dz_{2}.$$
That is the error is already $O(h)$. To continue improving we write
$$b(x_{1},y_{2},\xi_{2})=\sum_{k=0}^{\infty}h^{k}b_{k}(x_{1},y_{2},\xi_{2})$$
and successively solve transport equations to achieve an $O(h^{\infty})$ error. 
\end{proof}

\section*{Acknowledgements}
The author would like to acknowledgement the comments and suggestions of the reviewer.

\bibliography{references}

\begin{bibdiv}
\begin{biblist}

\bib{Feff83}{article}{
      author={Fefferman, Charles~L.},
       title={The uncertainty principle},
        date={1983},
        ISSN={0273-0979},
     journal={Bull. Amer. Math. Soc. (N.S.)},
      volume={9},
      number={2},
       pages={129\ndash 206},
         url={https://doi.org/10.1090/S0273-0979-1983-15154-6},
      review={\MR{707957}},
}

\bib{HTacy}{article}{
      author={Hassell, Andrew},
      author={Tacy, Melissa},
       title={Semiclassical {$L^p$} estimates of quasimodes on curved
  hypersurfaces},
        date={2012},
        ISSN={1050-6926},
     journal={J. Geom. Anal.},
      volume={22},
      number={1},
       pages={74\ndash 89},
         url={http://dx.doi.org/10.1007/s12220-010-9191-7},
}

\bib{KT98}{article}{
      author={Keel, Markus},
      author={Tao, Terence},
       title={Endpoint {S}trichartz estimates},
        date={1998},
        ISSN={0002-9327},
     journal={Amer. J. Math.},
      volume={120},
      number={5},
       pages={955\ndash 980},
  url={http://muse.jhu.edu/journals/american_journal_of_mathematics/v120/120.5keel.pdf},
      review={\MR{1646048}},
}

\bib{koch}{article}{
      author={Koch, Herbert},
      author={Tataru, Daniel},
      author={Zworski, Maciej},
       title={Semiclassical {$L\sp p$} estimates},
        date={2007},
        ISSN={1424-0637},
     journal={Ann. Henri Poincar\'e},
      volume={8},
      number={5},
       pages={885\ndash 916},
}

\bib{M16}{article}{
      author={Marshall, Simon},
       title={{$L^p$} norms of higher rank eigenfunctions and bounds for
  spherical functions},
        date={2016},
        ISSN={1435-9855},
     journal={J. Eur. Math. Soc. (JEMS)},
      volume={18},
      number={7},
       pages={1437\ndash 1493},
         url={https://doi.org/10.4171/JEMS/619},
      review={\MR{3506604}},
}

\bib{Mart02}{book}{
      author={Martinez, Andr\'{e}},
       title={An introduction to semiclassical and microlocal analysis},
      series={Universitext},
   publisher={Springer-Verlag, New York},
        date={2002},
        ISBN={0-387-95344-2},
         url={https://doi.org/10.1007/978-1-4757-4495-8},
      review={\MR{1872698}},
}

\bib{S06}{article}{
      author={Sarnak, Peter},
       title={Letter to {M}orawetz},
     journal={Available at http://www.math.princeton.edu/sarnak/},
}

\bib{S88}{article}{
      author={Sogge, Christopher~D.},
       title={Concerning the {$L^p$} norm of spectral clusters for second-order
  elliptic operators on compact manifolds},
        date={1988},
        ISSN={0022-1236},
     journal={J. Funct. Anal.},
      volume={77},
      number={1},
       pages={123\ndash 138},
         url={http://dx.doi.org/10.1016/0022-1236(88)90081-X},
}

\bib{S}{book}{
      author={Stein, Elias~M.},
       title={Harmonic analysis: real-variable methods, orthogonality, and
  oscillatory integrals},
      series={Princeton Mathematical Series},
   publisher={Princeton University Press, Princeton, NJ},
        date={1993},
      volume={43},
        ISBN={0-691-03216-5},
        note={With the assistance of Timothy S. Murphy, Monographs in Harmonic
  Analysis, III},
      review={\MR{1232192}},
}

\bib{tacy09}{article}{
      author={Tacy, Melissa},
       title={Semiclassical {$L^p$} estimates of quasimodes on submanifolds},
        date={2010},
        ISSN={0360-5302},
     journal={Comm. Partial Differential Equations},
      volume={35},
      number={8},
       pages={1538\ndash 1562},
         url={http://dx.doi.org/10.1080/03605301003611006},
}

\bib{Tacy18}{article}{
      author={Tacy, Melissa},
       title={A note on constructing families of sharp examples for {$L^p$}
  growth of eigenfunctions and quasimodes},
        date={2018},
        ISSN={0002-9939},
     journal={Proc. Amer. Math. Soc.},
      volume={146},
      number={7},
       pages={2909\ndash 2924},
         url={https://doi.org/10.1090/proc/14028},
      review={\MR{3787353}},
}

\bib{Tacy19}{article}{
      author={Tacy, Melissa},
       title={{$L^p$} estimates for joint quasimodes of semiclassical
  pseudodifferential operators},
        date={2019},
        ISSN={0021-2172},
     journal={Israel J. Math.},
      volume={232},
      number={1},
       pages={401\ndash 425},
         url={https://doi.org/10.1007/s11856-019-1878-2},
      review={\MR{3990947}},
}

\bib{Tom75}{article}{
      author={Tomas, Peter~A.},
       title={A restriction theorem for the {F}ourier transform},
        date={1975},
        ISSN={0002-9904},
     journal={Bull. Amer. Math. Soc.},
      volume={81},
       pages={477\ndash 478},
         url={https://doi.org/10.1090/S0002-9904-1975-13790-6},
      review={\MR{358216}},
}

\bib{Zworski12}{book}{
      author={Zworski, Maciej},
       title={Semiclassical analysis},
      series={Graduate Studies in Mathematics},
   publisher={American Mathematical Society},
     address={Providence, RI},
        date={2012},
      volume={138},
        ISBN={978-0-8218-8320-4},
}

\end{biblist}
\end{bibdiv}
\end{document}